\documentclass[a4paper,11pt]{amsart}
\usepackage{amsmath,amsfonts,amssymb,amsthm,enumerate}
\usepackage{mathrsfs}
\usepackage[pdftex]{graphicx}
\usepackage[all]{xy}
\usepackage{tikz}
\usepackage{float}
\usepackage{tikz-3dplot}

\usetikzlibrary{positioning}
\tikzset{>=stealth}

\theoremstyle{plain}
\newtheorem{thm}{Theorem}[section]
\newtheorem{lem}[thm]{Lemma}
\newtheorem{prop}[thm]{Proposition}
\newtheorem{cor}[thm]{Corollary}

\theoremstyle{definition}
\newtheorem{defn}[thm]{Definition}
 
\newtheorem{exmp}[thm]{Example} 
\newtheorem{notation}[thm]{Notation}
\newtheorem{rmk}[thm]{Remark}
\newtheorem{note}[thm]{Note}
\newtheorem{ques}[thm]{Question}
\newtheorem{caution}[thm]{Caution}

\newcommand{\OO}{\mathcal{O}}

\title{Slices of Okounkov bodies of big divisors on Mori dream spaces}
\author{Jaesun Shin}
\date{}

\address{Department of Mathematical Sciences, KAIST, 291 Daehak-ro, Yuseong-gu, Daejon 305-701, Korea}

\email{jsshin1991@kaist.ac.kr}

\begin{document}
\maketitle
\begin{abstract}
The purpose of this paper is to study the slices of the Okounkov bodies of Mori dream spaces. First, we analyze all the slices of the Okounkov bodies of big divisors on Mori dream spaces associated to some admissible flags. As a byproduct, we obtain their descriptions on Mori dream threefolds. Finally, we consider its application to the rational polyhedrality of them.
\end{abstract}

\begin{section} {Introduction}
In this paper, we work over $\mathbb{C}$ and assume that a variety is smooth unless otherwise stated. Let $X$ be a projective variety, $X_{\bullet}$ an admissible flag (see Definition \ref{defn:admissible flag}) and $D$ a big divisor on $X$. By the earlier works of Okounkov (\cite{O1}, \cite{O2}), Lazarsfeld and Musta\c t\v a (\cite{RM}) constructed the convex bodies of $D$ on $X$ associated to $X_{\bullet}$, deonted by $\Delta_{X_{\bullet}}(D)$, which we call the Okounkov bodies of $D$ associated to $X_{\bullet}$. The importance of the Okounkov bodies of a big divisor is that we can study geometric problems through combinatoric ones. After the construction by Lazarsfeld and Musta\c t\v a, there are many works to obtain the various properties of a divisor by using the Okounkov bodies. In particular, the Okounkov bodies are closely related to the numerical properties of a given big divisor. For example, it only depends on the numerical class of a big divisor (see \cite[Propositon 4.1]{RM}). Also,  K\" uronya and Lozovanu proved that nefness (\cite[Corollary 2.2]{KL2}) and ampleness (\cite[Corollary 3.2]{KL2}) of big divisors can be recovered from the shapes of the Okounkov bodies. Furthermore, they proved that the infinitesimal Okounkov bodies contain informations about moving Seshadri constant (\cite[Theorem C]{KL}). \

However, the explicit computations of the Okounkov bodies are difficult due to the complications unless $X$ is low-dimensional or has simple structures. For example, when $X$ is a curve, since a divisor $D$ is just a point, we can easily obtain $\Delta_{X_{\bullet}}(D)$ by using Riemann-Roch (see Example \ref{exmp:curve}). Also, when $X$ is a surface, Zariski proved that every pseudo-effective divisor $D$ has a Zariski decomposition (\cite[Proposition 2.3.19]{RM}). By using it, we can describe the Okounkov bodies of big divisors associated to a flag $X_{\bullet}:X \supset C \supset \{x\}$ (Example \ref{exmp:surface}) as the following shows (in fact, it can be extended to pseudo-effective divisors by using limiting Okounkov bodies (see \cite[Theorem 4.5]{CPW1})), 
\begin{align*}
\Delta_{X_{\bullet}}(D)=&\{(x_{1}, x_{2}) \in \mathbb{R}^{2}| \text{ } {\rm ord}_{C}(\lVert D \lVert) \le x_{1} \le \mu, \\
& \text{ } {\rm ord}_{x}(N_{D-x_{1}C}|_{C}) \le x_{2} \le \text{ } {\rm ord}_{x}(N_{D-x_{1}C}|_{C})+(C.P_{D-x_{1}C})\},
\end{align*}
where ${\rm ord}_{C}(\lVert D \lVert)$ is the asymptotic valuation of $D$ on $C$ (Definition \ref{defn:asymptotic valuation}), $\mu :={\rm sup}\{s>0|D-sC \text{ is big}\}$, $D-x_{1}C=P_{D-x_{1}C}+N_{D-x_{1}C}$ is the Zariski decomposition, and ${\rm ord}_{x}(N_{D-x_{1}C}|_{C})$ is the order of $N_{D-x_{1}C}|_{C}$ at $x$. \

Now, consider the case when $X$ is a variety of dimension $n \ge 3$. The following question arises from the above description on a surface.

\begin{ques} \label{ques:threefold}
Is there a general description of the Okounkov bodies of divisors on $X$ as that on a surface?
\end{ques}

One way to describe the Okounkov bodies is to analyze all the slices of them. In this way, we obtain the Okounkov bodies of a surface using the Zariski decomposition of a divisor. However, in higher dimensions, Zariski decompositions do not exist in general. Thus, when $X$ is a a variety of dimension $n \ge 3$, we can consider two cases. The first and the easiest case is when all the big divisors admit the Zariski decompositions. We can obtain their Okounkov bodies by following the same construction as that of a surface case. One of the simplest examples is a threefold whose pseudo-effective cone and nef cone coincide (see Example \ref{exmp:homogeneous}). In this case, by using the similar argument in the proof of Lemma \ref{lem:restriction}, we can describe the Okounkov bodies. The next difficult case is when not all divisors on $X$ have the Zariski decomposition. In general, it is hard to give an answer to Question \ref{ques:threefold} for the last case. Therefore, we consider it to the case when $X$ is a Mori dream space (see Definition \ref{defn:MDS}). One of the nice features of Mori dream spaces is that every divisor has a decomposition similar with a Zariski decomposition (see Proposition \ref{prop:Zariski decomposition}), which is helpful to obtain the Okounkov bodies. \

The aim of this paper is to analyze all the slices of Okounkov bodies of big divisors on Mori dream spaces, which are essential to have an answer on Question \ref{ques:threefold}. Now, let $X$ be a Mori dream space of dimension $n$, $D$ a big divisor on $X$, and $X_{\bullet}:X=Y_{0} \supset Y_{1} \supset \cdots \supset Y_{n}=\{p\}$ an admissible flag with $Y_{1} \cap (\cup_{i \in I_{D}} {\rm Ud}(f_{i}))=\emptyset$, where ${\rm Ud}(f_{i})$ and $I_{D}$ are defined in Note \ref{note:set-up}. The main idea is to observe $\Delta_{\tilde{X_{i}}_{\bullet}}(\phi_{i}^{*}D_{t})$ and $\Delta_{\tilde{X_{i}}_{\bullet}}(\tilde{f_{i}}^{*}({f_{i}}_{*}^{c}D_{t}))$, where $D_{t}$, $\phi_{i}$, ${\tilde{X_{i}}}_{\bullet}$, and $f_{i}$ are as in Notation \ref{notation:Mori dream threefold} and Note \ref{note:set-up}. By Proposition \ref{prop:important lemma}, we obtain that they are the same. By using this with some lemmas, we obtain our main theorem. 

\begin{thm}(=Theorem \ref{thm:slices}) \label{thm:introduction}
Let $X$, $X_{\bullet}$, $D$, $I_{D}$, and $[\alpha_{i}, \beta_{i}]$ be as in Notation \ref{notation:Mori dream threefold} and Note \ref{note:set-up}. Then, for each $i=1, \dots, r$, there exist the linear function $l_{i}(t)=(l_{i}^{1}(t), \cdots, l_{i}^{n-1}(t))$ defined on each $[\alpha_{i}, \beta_{i}]$ such that 
\begin{align*}
\Delta_{X_{\bullet}}(D)_{x_{1}=t}=\Delta_{{Y_{1}}_{\bullet}}(P_{D_{t}}|_{Y_{1}})+l_{i}(t)
\end{align*} 
for all $t=t_{i} \in [\alpha_{i}, \beta_{i}]$. 

\end{thm}

Note that the linear function $l_{i}(t)$ on Theorem \ref{thm:introduction} is defined in the proof of Lemma \ref{lem:restriction}. Theorem \ref{thm:introduction} says that if we know the Okounkov bodies of big divisors on $Y_{1}$, we obtain all the slices of $\Delta_{X_{\bullet}}(D)$ on $X$ so that $\Delta_{X_{\bullet}}(D)$ is known. Moreover, Theorem \ref{thm:introduction} can be extended to the limiting Okounkov bodies of pseudo-effective divisors naturally (Remark \ref{rmk:extension}). Since we know the Okounkov bodies of big divisors on a surface, we obtain their descriptions on Mori dream threefolds as a byproduct of Theorem \ref{thm:introduction}. 

\begin{cor} (=Corollary \ref{cor:Mori dream space}) \label{cor:introduction 2}
Let $X$ be a Mori dream threefold, and $X_{\bullet}:X \supset S \supset C \supset \{p\}$ an admissible flag with $S \cap (\cup_{i \in I_{D}} {\rm Ud}(f_{i}))=\emptyset$, where $D$, $I_{D}$, and $[\alpha_{i}, \beta_{i}]$ are as in Notation \ref{notation:Mori dream threefold} and Note \ref{note:set-up}. Then, for each $i=1, \dots, r$, there exist the linear function $l_{i}(t)=(l_{i}^{1}(t), l_{i}^{2}(t))$ defined on each $[\alpha_{i}, \beta_{i}]$ such that 
\begin{align*}
\Delta_{X_{\bullet}}(D)=&\{(x_{1},x_{2},x_{3}) \in \mathbb{R}^{3}| \text{ }\text{\rm ord}_{S}(\lVert D \lVert) \le x_{1} \le \mu, \text{ for each } x_{1}=t_{i} \in [\alpha_{i}, \beta_{i}], \\
& l_{i}^{1}(t_{i}) \le x_{2} \le \mu_{t_{i}}+l_{i}^{1}(t_{i}), \text{ } \delta_{x_{2}}(t_{i}) \le x_{3} \le \delta_{x_{2}}(t_{i})+(P_{(P_{D_{t_{i}}}|_{S}-x_{2}C)}.C)\},
\end{align*}
where $\mu={\rm sup}\{t>0|\text{ }D-tS \in {\rm Big}(X)\}$, $\mu_{t}={\rm sup}\{\alpha>0|\text{ }P_{D_{t}}|_{S}-\alpha C \in {\rm Big}(S)\}$, $P_{D_{t_{i}}}|_{S}-x_{2}C=P_{(P_{D_{t_{i}}}|_{S}-x_{2}C)}+N_{(P_{D_{t_{i}}}|_{S}-x_{2}C)}$ is the Zariski decomposition in the usual sense, and $\delta_{x_{2}}(t_{i})={\rm ord}_{p}(N_{(P_{D_{t_{i}}}|_{S}-x_{2}C)}|_{C})+l_{i}^{2}(t_{i})$ is a linear function on each $[\alpha_{i}, \beta_{i}]$ for fixed $x_{2} \in [l_{i}^{1}(t_{i}), \mu_{t_{i}}+l_{i}^{1}(t_{i})]$.
\end{cor}

Moreover, Corollary \ref{cor:pseudo-effective divisor} says that Corollary \ref{cor:introduction 2} also holds for a pseudo-effective divisor $D$. Then, the natural question one can ask is whether the description in Corollary \ref{cor:introduction 2} holds without the assumption that $S \cap (\cup_{i \in I_{D}} {\rm Ud}(f_{i}))=\emptyset$. However, we can see that it does not hold in general (Caution \ref{caution:caution}). \

The paper is organized as follows. In Section 2, we recall the construction of Okounkov bodies, that of the restricted Okounkov bodies, and some examples. Moreover, we recall the definition of a Mori dream space and their basic properties. Section 3 is the main part of this paper. In this section, we prove Theorem \ref{thm:introduction}, and describe the Okounkov bodies of big divisors on Mori dream threefolds. Finally, in Section 4, we give an application of Corollary \ref{cor:introduction 2} to the rational polyhedrality of the (limiting) Okounkov bodies of Mori dream threefolds. 

\begin{notation} \label{notation:usual} \text{ }
\begin{enumerate}[(1)]
\item ${\rm Pic}(X)$ : the group of Cartier divisors on $X$ modulo linear equivalence. 
\item ${\rm N^{1}}(X)$ : the group of Cartier divisors on $X$ modulo numerical equivalence. 
\item ${\rm Pic}(X)_{k} = {\rm Pic}(X) \otimes_{\mathbb{Z}} k$ for $k=\mathbb{Q}, \mathbb{R}$ and similarly for ${\rm N^{1}}(X)_{k}$.
\item ${\rm Amp}(X)$ : the cone in ${\rm N^{1}}(X)_{\mathbb{R}}$ spanned by ample divisors. 
\item ${\rm Nef}(X)$ : the cone in ${\rm N^{1}}(X)_{\mathbb{R}}$ spanned by nef divisors. 
\item ${\rm Mov}(X)$ : the cone in ${\rm N^{1}}(X)_{\mathbb{R}}$ spanned by movable divisors. 
\item ${\rm Big}(X)$ : the cone in ${\rm N^{1}}(X)_{\mathbb{R}}$ spanned by big divisors. 
\item ${\rm Eff}(X)$ : the cone in ${\rm N^{1}}(X)_{\mathbb{R}}$ spanned by effective divisors. 
\item $\overline{C}$ : the closure of the cone $C$. 
\item ${\rm Ud}(f)$ : the undefined locus of a rational map $f:X \dashrightarrow Y$ on $X$.
\end{enumerate}
\end{notation}

\end{section}

\textbf{Acknowledgements.}
I would like to thank my advisor Yongnam Lee, for his advice, encouragement and teaching. I also thank to Sung Rak Choi for helpful comments on limiting Okounkov bodies of pseudo-effective divisors. This work was supported by Basic Science Program through the National Research Foundation of Korea funded by the Korea government(MSIP)(No.2013006431).

\begin{section} {Okounkov bodies and Mori dream spaces}

In this section, we give some preliminaries which we need later on. We recall the definition, basic properties of the Okounkov bodies, and those of Mori dream spaces. 

\begin{subsection} {Contruction}

In this subsection, we define the Okounkov bodies of big divisors on a projective variety. 

\begin{defn} \label{defn:admissible flag}
Let $X$ be a projective variety of dimension $n$. Consider a complete flag 
\begin{align*}
X_{\bullet} : X=X_{0} \supseteq X_{1} \supseteq \dots \supseteq X_{n}=\text{\{point\}}
\end{align*}
of subvarieties of $X$, where $\text{codim}(X_{i})=i$ and each $X_{i}$ is smooth. We call this an admissible flag. 
\end{defn}

\begin{rmk}
In this paper, we assume that all flags are admissible. A divisor means $\mathbb{Z}$-divisor unless otherwise stated. 
\end{rmk}

Let $X$ be a projective variety of dimension $n$, $D$ a big divisor on $X$, and fix an admissible flag $X_{\bullet}$. Now, we define a valuation-like function 
\begin{align*}
\nu:=\nu_{X_{\bullet},D}:H^{0}(X, \OO_{X}(D)) \rightarrow \mathbb{Z}^{n} \cup \{\infty\}, s \mapsto \nu(s)=(\nu_{1}(s), \cdots, \nu_{n}(s))
\end{align*}
by the following way: \\

Let $0 \neq s \in H^{0}(X, \OO_{X}(D))$ be a nonzero section. Define 
\begin{align*}
\nu_{1}(s):={\rm ord}_{X_{1}}(s).
\end{align*}
After choosing a local equation $f$ for $X_{1}$ in $X$, $s$ determines a section 
\begin{align*}
\tilde{s_{1}}:=s \otimes f^{-\nu_{1}} \in H^{0}(X, \OO_{X}(D-\nu_{1}X_{1})).
\end{align*}
By restricting $\tilde{s_{1}}$ to $X_{1}$, we obtain
\begin{align*}
s_{1} \in H^{0}(X_{1}, \OO_{X_{1}}(D-\nu_{1}X_{1})).
\end{align*}
Then, $\nu_{2}(s):={\rm ord}_{X_{2}}(s_{1})$. Once we have defined $\nu_{i}(s)$ for $i \leq n-1$, we define inductively by the same way $s_{i}$ and $\nu_{i+1}$. The values $\nu_{i}$ define the function $\nu$ as desired.

\begin{note}
The above $\nu=\nu_{X_{\bullet}}$ satisfies three valuation-like properties (this is why we call $\nu$ a valuation-like function): 
\begin{enumerate}[(1)]
\item $\nu_{X_{\bullet}}(s)=\infty$ if and only if $s=0$. 
\item $\nu_{X_{\bullet}}(s_{1}+s_{2}) \ge \text{min} \{\nu_{X_{\bullet}}(s_{1}), \nu_{X_{\bullet}}(s_{2})\}$ by ordering $\mathbb{Z}^{n}$ lexicographically, where $s_{1}, s_{2} \in H^{0}(X,\OO_{X}(D))$ are non-zero sections. 
\item $\nu_{X_{\bullet},D+E}(s \otimes t)=\nu_{X_{\bullet},D}(s)+\nu_{X_{\bullet},E}(t)$ for non-zero sections $s \in H^{0}(X,\OO_{X}(D))$ and $t \in H^{0}(X,\OO_{X}(E))$.
\end{enumerate}
\end{note}

By using $\nu$, we can define Okounkov bodies. 

\begin{defn} \label{defn:definition of Okounkov bodies}
Let $X, D$ be as above. Define 
\begin{align*}
\Gamma(D)_{m}:={\rm Im}(((H^{0}(X, \OO_{X}(mD))-\{0\}) \overset{\nu}{\rightarrow} \mathbb{Z}^{n})
\end{align*}
for all $m \in \mathbb{N}$. Define
\begin{align*}
\Sigma(\Gamma)=\text{closure of the convex hull of }\{(\Gamma(D)_{m}, m), m \in \mathbb{N}\}.
\end{align*}
The Okounkov body of $D$ with respect to the fixed flag $X_{\bullet}$ is the compact convex set 
\begin{align*}
\Delta_{X_{\bullet}}(D):&=\Sigma(\Gamma) \cap (\mathbb{R}^{n} \times \{1\}) \\
&=\text{closed convex hull } (\bigcup_{m \ge 1} \frac{1}{m} \cdot \Gamma(D)_{m}) \subset \mathbb{R}^{n}.
\end{align*}
\end{defn}

\begin{rmk} \label{rmk:Q-divisor} For a big $\mathbb{Q}$-divisor $D$, $\Delta_{X_{\bullet}}(D):=\frac{1}{m}(\Delta_{X_{\bullet}}(mD))$, where $mD$ is an integral divisor. For definition of Okounkov bodies of big $\mathbb{R}$-divisors, see \cite[Subsection 2.2]{CHPW}. 
\end{rmk}

By Remark \ref{rmk:Q-divisor}, we can define limiting Okounkov body of a pseudo-effective divisor, which is a natural generalization of Definitioin \ref{defn:definition of Okounkov bodies}. 

\begin{defn} 
Let $D$ be a pseudo-effective divisor on $X$. The limiting Okounkov body $\Delta_{X_{\bullet}}^{\rm lim}(D)$ of $D$ associated to the flag $X_{\bullet}$ is defined as 
\begin{align*}
\Delta_{X_{\bullet}}^{\rm lim}(D):=\lim_{\epsilon \rightarrow 0}\Delta_{X_{\bullet}}(D+\epsilon A)=\bigcap_{m \in \mathbb{N}}\Delta_{X_{\bullet}}(D+\frac{1}{m}A),
\end{align*}
where $A$ is a nef and big divisor on $X$. By the continuity of the Okounkov bodies, when $D$ is big, $\Delta_{X_{\bullet}}^{\rm lim}(D)=\Delta_{X_{\bullet}}(D)$. 
\end{defn}

\end{subsection}

\begin{subsection} {Restricted Okounkov bodies}

We recall the definition of the restricted Okounkov bodies. Before going on, we first define the restricted complete linear series of a divisor. 

\begin{defn} (Restricted complete linear series)
Let $X$ be a projective variety and $V$ a subvariety of $X$. Let $D$ be a big divisor on $V$. Set 
\begin{align*}
W_{m}=H^{0}(X|V, \OO_{X}(mD)):={\rm Im}(H^{0}(X,\OO_{X}(mD)) \overset{\rm rest.}{\longrightarrow} H^{0}(V, \OO_{V}(mD)),
\end{align*}
where the map is the restriction map. We call such a graded linear series $\{W_{m}\}_{m \in \mathbb{Z}_{\ge 0}}$ on $X$ the restricted complete linear series of $D$ from $X$ to $V$. 
\end{defn}

Now, we can define the restricted Okounkov bodies.   

\begin{defn} \label{defn:restricted}
Let $X$ be a projective variety and $V$ a subvariety of $X$. Given an admissible flag $X_{\bullet}$, we call $\Delta_{X_{\bullet}|V}(D)$ the restricted Okounkov bodies of $D$ to $V$ if we take sections in the restricted complete linear series of $D$ to $V$. 
\end{defn}

The importance of the restricted Okounkov bodies comes from the fact that it is closely related to the restriction of the Okounkov bodies of a divisor (Proposition \ref{prop:restriction}).

\begin{defn} \cite[Definition 2.1.20, 10.3.2]{RL}
Let $X$ be a projective variety, and $L$ a big divisor. The stable base locus of $L$ is the algebraic set
\begin{align*}
{\rm \bf{B}}(L):=\bigcap_{m \ge 1} {\rm Bs}(\lvert mL \rvert).
\end{align*}
The augmented base locus and restricted base locus of $L$ are 
\begin{align*}
{\rm \bf{B}}_{+}(L):=\bigcap_{A}{\rm \bf{B}}(L-A) \text{, and } {\rm \bf{B}}_{-}(L):=\bigcup_{A} {\rm \bf{B}}(L+A),
\end{align*}
respectively, where the intersection and the union are taken over all ample divisors $A$.  
\end{defn}

\begin{note}
${\rm \bf{B}}_{+}(L)$ and ${\rm \bf{B}}_{-}(L)$ are independent of the choice of $A$ and sufficiently small $\epsilon>0$.
\end{note}

\begin{prop} \label{prop:restriction}
Let $X$ be a projective variety of dimension $n$ and $X_{\bullet}:X \supset X_{1}=E \supset X_{2} \supset \dots \supset X_{n}=\{point\}$ a given flag. Let $D$ be a big divisor on $X$ and suppose that $E \nsubseteq {\rm \bf{B}}_{+}(D)$. For $t \in \mathbb{R}$ with $\Delta_{X_{\bullet}}(D)_{x_{1}=t} \neq \emptyset$, 
\begin{align*}
\Delta_{X_{\bullet}}(D)_{x_{1}=t}=\Delta_{X_{\bullet}|E}(D-tE).
\end{align*}
\end{prop}

Proposition \ref{prop:restriction} plays a central role in getting the Okounkov bodies of a big divisor on a surface. However, computing the restricted Okounkov body is another problem. In an ample divisor case, we can have an easy answer since the first cohomology vanishes by Serre vanishing. For details, see \cite[Proposition 3.1]{KLM}. 

\begin{prop} \label{prop:restricted ample}
Let $X$ and $X_{\bullet}$ be as in Proposition \ref{prop:restriction}. Let $A$ be an ample divisor on $X$. Then, 
\begin{align*}
\Delta_{X_{\bullet}|X_{1}}(A)=\Delta_{{X_{1}}_{\bullet}}(A|_{X_{1}}).
\end{align*}
\end{prop}

\end{subsection}

\begin{subsection} {Examples}

In this subsection, we show some examples of the Okounkov bodies of big divisors on a curve, surface, and some special varieties. 

\begin{exmp} \label{exmp:curve} (Curve) \

Let $C$ be a smooth projective curve of genus $g$. Fix a point $p \in C$ and consider an admissible flag $C_{\bullet} : C \supseteq p$. Let $D$ be a divisor on $C$. Consider a valuation 
\begin{align*}
\nu_{C_{\bullet}}:H^{0}(C, \OO_{C}(D)) \rightarrow \mathbb{Z}.
\end{align*}
In this case, it is just ${\rm ord}_{p}(s)$ for $s \in H^{0}(C,\OO_{C}(D))$. Let $d:={\rm deg}_{C}(D)$. Then, it is clear that $\Delta_{C_{\bullet}}(D) \subseteq [0,d]$. Now, we show that the reverse inclusion holds. For $m \gg 0$ such that $md \ge 2g$, $|mD|$ is base point free. Thus, ${\rm ord}_{p}(s)=0$ for some $s \in H^{0}(C, \OO_{C}(D))$. Moreover, for such $m \gg 0$, by Riemann-Roch, $h^{0}(mD)=md+1-g$. It is well-known that 
\begin{align*}
\text{the number of points in } \nu_{C_{\bullet}}(H^{0}(mD)) = h^{0}(mD)=md+1-g
\end{align*}
for all $m \gg 0$. So convex hull of $\nu_{C_{\bullet}}(H^{0}(mD))$ contains $[0, md-g]$ for all $m \gg 0$: i.e., $[0, d-\frac{g}{m}] \subseteq \Delta_{C_{\bullet}}(D)$ for all $m \gg 0$. Therefore, $\Delta_{C_{\bullet}}(D)=[0,d]$. 
\end{exmp}

As in Example \ref{exmp:curve}, the Okounkov bodies of a divisor in a curve is easy to obtain. However, even in surface cases, it becomes much more complicated. 

\begin{exmp} \label{exmp:surface} (Surface) \

We refer the readers to \cite[Section 6]{RM} for more details. Let $X$ be a surface and $X_{\bullet}:X \supset C \supset x$ an admissible flag. The main idea is the following lemma. 
\begin{lem}
Let $D$ be a big $\mathbb{Q}$-divisor on a surface $X$ with a Zariski decomposition $D=P+N$. Assume that $C \nsubseteq \textbf{B}_{+}(D)$. Let 
\begin{align*}
\alpha(D)={\rm ord}_{x}(N|_{C}), \text{  } \beta(D)={\rm ord}_{x}(N|_{C})+(C.P).
\end{align*}
Then, the restricted Okounkov body of $D$ is 
\begin{align*}
\Delta_{X_{\bullet}|C}(D)=[\alpha(D), \beta(D)].
\end{align*}
\end{lem}

Before going on, we need to define the asymptotic valuation ${\rm ord}_{V}(\lVert D \lVert)$. For more details, see \cite[Subsection 2.1]{CHPW}. 

\begin{defn} \label{defn:asymptotic valuation}
Let $X$ be a projective variety and $V \subset X$ be an irreducible subvariety of $X$. For a big divisor $D$ on $X$, define the asymptotic valuation of $D$ on $V$ as 
\begin{align*}
{\rm ord}_{V}(\lVert D \rVert):={\rm inf}\{ {\rm ord}_{V}(D')| \text{ } D \equiv D' \ge 0\}.
\end{align*}
\end{defn}

Now, $\mu :={\rm sup}\{s>0|D-sC \text{ is big}\}$. Then, we obtain the following proposition by using the above lemma with the linearity property of Zariski chambers in \cite{BKS}. See \cite[Theorem 6.4]{RM} for a proof. 

\begin{prop} 
Let $X$, $X_{\bullet}$, and $D$ be as above. Then, 
\begin{align*}
\Delta_{X_{\bullet}}(D)=&\{(x_{1}, x_{2}) \in \mathbb{R}^{2}| \text{ } {\rm ord}_{C}(\lVert D \lVert) \le x_{1} \le \mu, \\
& \text{ } {\rm ord}_{x}(N_{D-x_{1}C}|_{C}) \le x_{2} \le \text{ } {\rm ord}_{x}(N_{D-x_{1}C}|_{C})+(C.P_{D-x_{1}C})\},
\end{align*}
where ${\rm ord}_{C}(\lVert D \lVert)$ is the asymptotic valuation of $D$ on $C$, $D-x_{1}C=P_{D-x_{1}C}+N_{D-x_{1}C}$ is the Zariski decomposition, and ${\rm ord}_{x}(N_{D-x_{1}C}|_{C})$ is the order of $N_{D-x_{1}C}|_{C}$ at $x$.
\end{prop}

\end{exmp}

\begin{rmk}
As you can see in the proof of \cite[Theorem 6.4]{RM}, the main idea is the existence of Zariski decompositions on a surface. However, it is not the case in higher dimensions. Thus, it is difficult to know the general expressions of Okounkov bodies in higher dimensional cases. 
\end{rmk}

The following example is the simple case in higher dimensions. For notations, see Notation \ref{notation:usual}.

\begin{exmp} (Threefold $X$ with $\overline{{\rm Eff}}(X)={\rm Nef}(X)$) \label{exmp:homogeneous}
\

Let $X$ be a threefold with $\overline{{\rm Eff}}(X)={\rm Nef}(X)$. Then, the big cone and the ample cone coincide. Let $X_{\bullet}:X \supset S \supset C \supset \{p\}$ be an admissible flag, where $S$ satisfies ${\overline{{\rm Eff}}(S)}={\rm Nef}(S)$. Then, \cite[Corollary 3.2]{KLM} says that for any big divisor $D$ (in this case, $D$ is ample),  
\begin{align*}
\Delta_{X_{\bullet}}(D)=&\{(x_{1},x_{2},x_{3}) \in \mathbb{R}^{3}| \text{ } 0 \le x_{1} \le \mu, \text{ } 0 \le x_{2} \le \mu_{x_{1}}, \text{ }0 \le x_{3} \le ({(D-x_{1}S)}|_{S}-x_{2}C.C)\},
\end{align*} 
where $\mu={\rm sup} \{t>0| D-tS \in {\rm Amp}(X)\}$, $\mu_{x_{1}}={\rm sup} \{\alpha>0|(D-x_{1}S)|_{S}-\alpha C \in {\rm Amp}(X)\}$. \

However, if ${\overline{{\rm Eff}}(S)}={\rm Nef}(X)$ does not hold, then by using \cite[Proposition 3.1]{KLM} with Example \ref{exmp:surface}, we can describe $\Delta_{X_{\bullet}}(D)$. 
\end{exmp}

\end{subsection}

\begin{subsection} {Mori dream spaces}

In this subsection, we briefly recall some definitions and results related to Mori dream spaces. 

\begin{defn}
Let $X$ be a normal projective variety. 
\begin{enumerate}[(1)]
\item A small $\mathbb{Q}$-factorial modification(SQM) of $X$ is a small birational map $\mu:X \dashrightarrow Y$ to another normal $\mathbb{Q}$-factorial projective variety $Y$. 
\item A contracting birational map is a birational map surjective in codimension one. 
\end{enumerate}
\end{defn}

\begin{defn} \label{defn:MDS}
A normal projective $\mathbb{Q}$-factorial variety $X$ is called a Mori dream space (MDS) if the following conditions hold. 
\begin{enumerate}[(1)]
\item ${\rm Pic}(X)_{\mathbb{Q}}={\rm N^{1}}(X)_{\mathbb{Q}}$.
\item ${\rm Nef}(X)$ is the affine hull of finitely many semi-ample divisors. 
\item There are finitely many small $\mathbb{Q}$-factorial modifications $f_{i}:X \dashrightarrow X_{i}$ such that each $X_{i}$ satisfies 1), 2) and ${\rm Mov}(X)= \bigcup(f_{i}^{*}({\rm Nef(X_{i})}))$. 
\end{enumerate}
\end{defn}

The following proposition is a basic fact relating to Mori dream spaces. See \cite[Proposition 1.11]{HK} for details. 

\begin{prop} \label{prop:basic}
Let $X$ be a Mori dream space. 
\begin{enumerate}[(1)]
\item The $f_{i}$ in Definition \ref{defn:MDS}-(3) are all the small $\mathbb{Q}$-factorial modifications of $X$. In particular, the identity of $X$ must appear among them.
\item There are finitely many contracting birational maps of $X$, $g_{i}:X \dashrightarrow Y_{i}$, with $Y_{i}$ a Mori dream space, such that 
\begin{align*}
{\rm Eff}(X) = \bigcup_{i}(g_{i}^{*}({\rm Nef(Y_{i})}) * {\rm Ex}(g_{i}))
\end{align*}
gives a decomposition of ${\rm Eff}(X)$ into closed rational polyhedral subcones (which we call them Mori chambers) with disjoint interiors, where ${\rm Ex}(g_{i})$ denotes the cone spanned by the irreducible components of the exceptional locus of $g_{i}$ and $*$ denotes the join. 
\end{enumerate}
\end{prop}

Now, we recall definitions of a Zariski decomposition and related propositions of Mori dream spaces using Proposition \ref{prop:basic}. 

\begin{defn} \label{defn:Zariski decomposition}
Let $X$ be a normal projective variety and $D$ a pseudo-effective $\mathbb{Q}$-divisor on $X$. A Zariski decomposition of $D$ consists of a nef $\mathbb{Q}$-divisor $P$ and an effective $\mathbb{Q}$-divisor $N$ such that $D=P+N$ and for all sufficiently divisible positive integer $m$ with $mD$, $mP$ integral, the natural map 
\begin{align*}
H^{0}(X, \OO_{X}(mP)) \rightarrow H^{0}(X, \OO_{X}(mD))
\end{align*}
is an isomorphism, where the map is the multiplication of a canonical section of $\OO_{X}(mN)$. \\

Moreover, if $X$ is a Mori dream space, we say that $D=P+N$ is a Zariski decomposition of $D$ in terms of MDS if there exists a small $\mathbb{Q}$-factorial modification $f:X \dashrightarrow X'$ such that 
\begin{align*}
f^{c}_{*}D=f^{c}_{*}P+f^{c}_{*}N
\end{align*}
is a Zariski decomposition in the above sense, where $f^{c}_{*}$ is the cycle pushforward of codimension 1 cycles. 
\end{defn}

One of the most important features of a divisor on a smooth projective surface is that it has a Zariski decomposition. In general, in higher dimensions, we cannot say that. However, the following proposition says that in Mori dream spaces, we can say the similar one. See \cite[Proposition 2.13]{O} for details. 

\begin{prop} \label{prop:Zariski decomposition}
Let $X$ be a Mori dream space. Consider the decomposition of ${\rm Eff}(X)$ in Proposition \ref{prop:basic}. Then, for each chamber $C$, there exists a small $\mathbb{Q}$-factorial modification $f_{i}:X \dashrightarrow X_{i}$ of $X$ and two $\mathbb{Q}$-linear maps 
\begin{align*}
P, N:C \rightarrow {\rm Eff}(X)
\end{align*}
such that for any integral $D \in C$, $D=P(D)+N(D)$ gives a Zariski decomposition of $D$ as a divisor on $X_{i}$. \

Conversely, let $X$ be a $\mathbb{Q}$-factorial normal projective variety such that ${\rm Pic}(X)_{\mathbb{Q}} \cong {\rm N^{1}}(X)_{\mathbb{Q}}$. Assume that ${\rm Eff}(X)$ is decomposed into finitely many chambers $C$ on each of which there exists $\mathbb{Q}$-linear Zariski decompositions in the above sense, with positive parts semi-ample on some small $\mathbb{Q}$-factorial modifications of $X$. Then, $X$ is a Mori dream space.
\end{prop}

\begin{rmk}
We call $C$ in Proposition \ref{prop:Zariski decomposition} as Mori chamber and ${\rm Eff}(X)=\bigcup C$ a Mori chamber decomposition.
\end{rmk}

\end{subsection}

\end{section}

\begin{section} {Slices of Okounkov bodies of Mori dream spaces}

When we deal with the Okounkov bodies in higher dimensions, the main obstruction is that we cannot ensure that every divisor has a Zariski decomposition. It does not happen in general since nef cone and movable cone are not the same in higher dimensions. This poses the difficulty in finding the restricted Okounkov bodies, which makes it hard to find the Okounkov bodies of a given divisor. \

However, in Mori dream spaces, although not all pseudo-effective divisors may have Zariski decompositions in the usual sense, Proposition \ref{prop:Zariski decomposition} implies that all the divisors have a decomposition that is similar with a Zariski decomposition. In this section, we analyze all the slices of Okounkov bodies of Mori dream spaces. As a byproduct, we obtain the description of Okounkov bodies of Mori dream threefolds. Now, we start it by fixing some notations. 

\begin{note}
In this paper, a divisor means an integral divisor. However, all the arguments in this section can be extended to $\mathbb{Q}$-divisors by the homogeneity of the Okounkov body. Moreover, by the continuity of the Okounkov body, we can extend them to $\mathbb{R}$-divisors. We leave the details to the reader. 
\end{note}

\begin{notation} \label{notation:Mori dream threefold}
Let $X$ be a Mori dream space of dimension $n$, $f_{i}:X \dashrightarrow X_{i}$ SQMs of $X$, $D$ a big divisor on $X$, and $X_{\bullet}:X=Y_{0} \supset Y_{1} \supset \cdots \supset Y_{n}=\{p\}$ an admissible flag with $Y_{1} \cap (\cup_{i \in I_{D}} {\rm Ud}(f_{i}))=\emptyset$, where $I_{D}$ is defined on Note \ref{note:set-up}. Finally, let ${\rm Eff}(X)=\bigcup_{i=1}^{r} C_{i}$ be a Mori chamber decomposition of ${\rm Eff}(X)$. 
\end{notation}

\begin{lem} \label{lem:birational}
Let $X$ be a normal, projective variety of dimension $n$ and $D$ a pseudo-effective divisor on $X$. Let $\mu:X' \rightarrow X$ be a projective birational morphism. Then, 
\begin{align*}
\Delta_{X_{\bullet}}^{\rm lim}(D)=\Delta_{X'_{\bullet}}^{\rm lim}(\mu^{*}D)
\end{align*}
for any admissible flag $X_{\bullet}:X_{0}=X \supset X_{1} \supset \cdots \supset X_{n}$ with $X_{n}$ not contained in the center of the exceptional locus of $\mu$, and the admissible flag $X'_{\bullet}$ induced by the strict transform of $X_{\bullet}$. 
\end{lem}

\begin{proof}
Let $\mu:X' \rightarrow X$ be a projective birational morphism. Since $X_{n}$ is not contained in the center of the exceptional locus of $\mu$, we can define $X'_{\bullet}$ using strict transform of $X_{\bullet}$. First, let $D$ be a big divisor. Since $X$ is normal, $H^{0}(X, \OO_{X}(mD))=H^{0}(X', \OO_{X'}(\mu^{*}mD))$ for all $m \in \mathbb{N}$. Therefore, we conclude that $\Delta_{X_{\bullet}}(D)=\Delta_{X'_{\bullet}}(\mu^{*}D)$. \

In general, let $D$ be a pseudo-effective divisor, $A$ a nef and big divisor on $X$. Since $D+\frac{1}{m}A$ is big for all $m \in \mathbb{N}$,  
\begin{align*}
\Delta_{X_{\bullet}}^{\rm lim}(D)=\bigcap_{m \in \mathbb{N}} \Delta_{X_{\bullet}}(D+\frac{1}{m} A)=\bigcap_{m \in \mathbb{N}} \Delta_{X'_{\bullet}}(\mu^{*}D+\frac{1}{m} \mu^{*}A) = \Delta_{X'_{\bullet}}^{\rm lim}(\mu^{*}D).
\end{align*}`
\end{proof}

\begin{note} (Set-up and notation) \label{note:set-up} \\
Notation is as in Notation \ref{notation:Mori dream threefold}. Consider $\{t \in \mathbb{R}_{\ge 0} | D-tY_{1} \in \bar{C_{i}}\}$, denoted by $[\alpha_{i}, \beta_{i}]$. (Note : It is just one closed interval since each Mori chamber is convex.) Consider the elimination of indeterminacy of $f_{i}$, 
\begin{displaymath}
\xymatrix{ \tilde{X_{i}} \ar[dr]^{\tilde{f_{i}}} \ar[d]^{\phi_{i}} & \\
X \ar@{-->}[r]^{f_{i}} & X_{i}}
\end{displaymath}
Now, let $D_{t}:=D-tY_{1}$ for $t \in \mathbb{R}_{\ge 0}$, and $\tilde{X_{i}}_{\bullet}: \tilde{X_{i}}=\tilde{Y_{0}}^{i} \supset \tilde{Y_{1}}^{i} \supset \cdots \supset \tilde{Y_{n}}^{i}$ a flag induced by strict transforms of $X_{\bullet}$ using $\phi_{i}$ (It makes sense since $Y_{1} \cap (\cup_{i \in I_{D}} {\rm Ud}(f_{i}))=\emptyset$ holds). Let $\mu={\rm sup}\{t>0|\text{ }D-tY_{1} \in {\rm Big}(X)\}$. Let $D_{t}=P_{D_{t}}+N_{D_{t}}$ be its Zariski decomposition in terms of MDS, and $I_{D}:=\{i\text{ }|\text{ } D_{t} \in C_{i}\text{ for some $t \in [{\rm ord}_{Y_{1}}(\lVert D \lVert), \mu]$}\}$, where $\mu={\rm sup}\{t>0|\text{ }D-tY_{1} \in {\rm Big}(X)\}$. 
\end{note}

Before going on, we need some lemmas. 

\begin{lem} \label{lem:restricted pullback}
Notation is as in Notation \ref{notation:Mori dream threefold} and Note \ref{note:set-up}. Then, $\Delta_{X_{\bullet}|Y_{1}}(D_{t})=\Delta_{\tilde{X_{i}}_{\bullet}|\tilde{Y_{1}}^{i}}(\phi_{i}^{*}D_{t})$ for all $t \in [{\rm ord}_{Y_{1}}(\lVert D \lVert), \mu]$. 
\end{lem}

\begin{proof}
By Lemma \ref{lem:birational}, we have $\Delta_{X_{\bullet}}(D_{t})=\Delta_{\tilde{X_{i}}_{\bullet}}(\phi_{i}^{*}D_{t})$. Therefore, by \cite[Theorem 4.24]{RM},  
\begin{align*}
\Delta_{X_{\bullet}|Y_{1}}(D_{t})=\Delta_{X_{\bullet}}(D_{t})_{x_{1}=0}=\Delta_{\tilde{X_{i}}_{\bullet}}(\phi_{i}^{*}D_{t})_{x_{1}=0}=\Delta_{\tilde{X_{i}}_{\bullet}|\tilde{Y_{1}}^{i}}(\phi_{i}^{*}D_{t}).
\end{align*}
\end{proof}

\begin{lem} \label{lem:restriction}
Notation is as in Notation \ref{notation:Mori dream threefold} and Note \ref{note:set-up}. Then, for each $i=1, \dots, r$, there exist the linear function $l_{i}(t)=(l^{1}_{i}(t), \dots, l^{n-1}_{i}(t))$ defined on each $[\alpha_{i}, \beta_{i}]$ such that
\begin{align*}
\Delta_{\tilde{X_{i}}_{\bullet}|\tilde{Y_{1}}^{i}}(\tilde{f_{i}}^{*}({f_{i}}_{*}^{c}D_{t}))=\Delta_{\tilde{X_{i}}_{\bullet}|\tilde{Y_{1}}^{i}}(\tilde{f_{i}}^{*}({f_{i}}_{*}^{c}P_{D_{t}}))+l_{i}(t).
\end{align*}
\end{lem}

\begin{proof}
Note that $X_{i}$ is normal. Also, since $\tilde{f_{i}}$ is birational and proper, by Zariski main theorem and Stein factorization, ${f_{i}}_{*}\OO_{\tilde{X_{i}}}=\OO_{X_{i}}$. Now, consider the following isomorphisms. 
\begin{align*}
H^{0}(\tilde{X_{i}}, \tilde{f_{i}}^{*}({f_{i}}_{*}^{c}P_{D_{t}})) \overset{g_{1}}{\underset{\cong}{\longrightarrow}} H^{0}(X_{i}, {f_{i}}_{*}^{c}P_{D_{t}}) \overset{g_{2}}{\underset{\cong}{\longrightarrow}} H^{0}(X_{i}, {f_{i}}_{*}^{c}D_{t}) \overset{g_{3}}{\underset{\cong}{\longrightarrow}} H^{0}(\tilde{X_{i}}, \tilde{f_{i}}^{*}({f_{i}}_{*}^{c}D_{t})).
\end{align*}

The middle isomorphism ($g_{2}$) comes from the Zariski decomposition property of a Mori dream space, and the first ($g_{1}$) and the last ($g_{3}$) ones come from the projection formula with $\tilde{f_{i}}_{*}\OO_{\tilde{X_{i}}}=\OO_{X_{i}}$: i.e.,
\begin{align*}
{f_{i}}_{*}^{c}D_{t}=\tilde{f_{i}}_{*}\OO_{\tilde{X_{i}}} \otimes {f_{i}}_{*}^{c}D_{t} \rightarrow \tilde{f_{i}}_{*}(\OO_{\tilde{X_{i}}} \otimes \tilde{f_{i}}^{*}({f_{i}}_{*}^{c}D_{t}))=\tilde{f_{i}}_{*}\tilde{f_{i}}^{*}({f_{i}}_{*}^{c}D_{t})),
\end{align*}
which is just the pull-back section map when we take global sections. Thus, 
\begin{align*}
\tilde{f_{i}}^{*}s \overset{g_{1}}{\longmapsto} s \overset{g_{2}}{\longmapsto} s \otimes s_{N_{D_{t}}} \overset{g_{3}}{\longmapsto} \tilde{f_{i}}^{*}s \otimes \tilde{f_{i}}^{*}s_{N_{D_{t}}},
\end{align*} 
where $s_{N_{D_{t}}} \in H^{0}(X_{i}, {f_{i}}_{*}^{c}N_{D_{t}})$ is the canonical section. It also holds when we consider $m\tilde{f_{i}}^{*}({f_{i}}_{*}^{c}D_{t})$, $m{f_{i}}_{*}^{c}D_{t}$, $m{f_{i}}_{*}^{c}P_{D_{t}}$, and $m\tilde{f_{i}}^{*}({f_{i}}_{*}^{c}P_{D_{t}})$ by taking $s_{N_{D_{t}}}^{\otimes m} \in H^{0}(X_{i}, m{f_{i}}_{*}^{c}N_{D_{t}})$ for $m \ge 1$.  Therefore, since pull-back functor commutes with $\otimes$, for $m \ge 1$,  
\begin{align*}
&H^{0}(\tilde{X_{i}}|\tilde{Y_{1}}^{i}, m\tilde{f_{i}}^{*}({f_{i}}_{*}^{c}P_{D_{t}}))=\{(\tilde{f_{i}}^{*}s)|_{\tilde{Y_{1}}^{i}} \text{ } | \text{ } s \in H^{0}(X_{i}, m{f_{i}}_{*}^{c}P_{D_{t}})\}, \\
&H^{0}(\tilde{X_{i}}|\tilde{Y_{1}}^{i}, m\tilde{f_{i}}^{*}({f_{i}}_{*}^{c}D_{t}))=\{(\tilde{f_{i}}^{*}s)|_{\tilde{Y_{1}}^{i}} \otimes (\tilde{f_{i}}^{*}s_{N_{D_{t}}})^{\otimes m}|_{\tilde{Y_{1}}^{i}} \text{ } | \text{ } s \in H^{0}(X_{i}, m{f_{i}}_{*}^{c}P_{D_{t}})\}.
\end{align*} 

Now, let $l_{i}(t)=\nu_{\tilde{Y_{1}}^{i}_{\bullet}}((\tilde{f_{i}}^{*}s_{N_{D_{t}}})|_{\tilde{Y_{1}}^{i}})$. Since $\nu_{\tilde{Y_{1}}^{i}_{\bullet}}(s_{1} \otimes s_{2})=\nu_{\tilde{Y_{1}}^{i}_{\bullet}}(s_{1})+\nu_{\tilde{Y_{1}}^{i}_{\bullet}}(s_{2})$, we obtain the desired result except that $l_{i}(t)$ is linear on $[\alpha_{i}, \beta_{i}]$. However, the linearity comes from the linearity of $N_{D_{t}}$ in each Mori chamber (see Proposition \ref{prop:Zariski decomposition}) and the well-known fact that $s_{D}=s_{D_{1}} \otimes s_{D_{2}}$, where $s_{D}$, $s_{D_{i}}$ are the canonical sections of effective divisors $D$ and $D_{i}$ such that $D=D_{1}+D_{2}$, which proves the lemma. 
\end{proof}

\begin{rmk}
In fact, we do not have to assume `$Y_{1} \cap (\cup_{i \in I_{D}} {\rm Ud}(f_{i}))= \emptyset$' in Lemma \ref{lem:restriction}. However, this assumption is essential for Proposition \ref{prop:important lemma}. 
\end{rmk}

\begin{prop} \label{prop:important lemma}
Notation is as in Notation \ref{notation:Mori dream threefold} and Note \ref{note:set-up}. Then,  
\begin{align*}
\Delta_{\tilde{X_{i}}_{\bullet}}(\phi_{i}^{*}D_{t})=\Delta_{\tilde{X_{i}}_{\bullet}}(\tilde{f_{i}}^{*}({f_{i}}_{*}^{c}D_{t}))
\end{align*}
for each $i=1, \cdots r$. 
\end{prop}

\begin{proof}
Note that $Y_{1} \cap (\cup_{i \in I_{D}} {\rm Ud}(f_{i}))=\emptyset$. Let $U \subset X$ and $U_{i} \subset X_{i}$ be isomorphic loci of $\phi_{i}, f_{i}$, and $\tilde{f_{i}}$. Now, consider strict transforms of $\tilde{f_{i}}^{*}({f_{i}}_{*}^{c}D_{t})$ and $\phi_{i}^{*}D_{t}$, denoted by $\tilde{f_{i}}_{*}^{-1}({f_{i}^{c}}_{*}D_{t})$ and ${\phi_{i}}_{*}^{-1}D_{t}$, respectively. Note that 
\begin{align*}
\tilde{f_{i}}_{*}^{-1}({f_{i}^{c}}_{*}D_{t})&=\overline{\tilde{f_{i}}^{-1}(\overline{f_{i}(D_{t} \cap U)} \cap U_{i})}=\overline{\phi_{i}^{-1} \circ f_{i}^{-1}(\overline{f_{i}(D_{t} \cap U)} \cap U_{i})} \\
&=\overline{\phi_{i}^{-1} \circ f_{i}^{-1}(f_{i}(D_{t} \cap U) \cap U_{i})}=\overline{\phi_{i}^{-1}(D_{t} \cap U)}. 
\end{align*}
Also, since ${\phi_{i}}_{*}^{-1}D_{t}=\overline{\phi_{i}^{-1}(D_{t} \cap U)}$, we obtain that their strict transforms are the same. By the assumption, $\tilde{f_{i}}^{*}({f_{i}}_{*}^{c}D_{t})|_{\tilde{Y_{1}}^{i}} = {\phi_{i}}^{*}D_{t}|_{\tilde{Y_{1}}^{i}}$. 
Thus, we obtain the following diagram 
\begin{displaymath}
\xymatrix{ H^{0}(\tilde{X_{i}}, \phi_{i}^{*}D_{t}) \ar[r]^{\psi_{1}} & H^{0}(\tilde{Y_{1}}^{i}, \phi_{i}^{*}D_{t}|_{\tilde{Y_{1}}^{i}}) \ar@{=}[d] \\
H^{0}(\tilde{X_{i}}, \tilde{{f_{i}}}^{*}({f_{i}}_{*}^{c}D_{t})) \ar[r]^{\psi_{2}} & H^{0}(\tilde{Y_{1}}^{i}, \tilde{f_{i}}^{*}({f_{i}}_{*}^{c}D_{t})|_{\tilde{Y_{1}}^{i}}), }
\end{displaymath}
where $\psi_{1}$ and $\psi_{2}$ are restriction maps. Since the strict transforms of $\tilde{{f_{i}}}^{*}({f_{i}}_{*}^{c}D_{t})$ and $\phi_{i}^{*}D_{t}$ are exactly the same, and $\tilde{Y_{1}}^{i} \cap E = \emptyset$ for any exceptional divisor $E$ by the assumption $Y_{1} \cap (\cup_{i \in I_{D}} {\rm Ud}(f_{i}))=\emptyset$, the images of $\psi_{1}$ and $\psi_{2}$ are the same. Thus, we obtain $\Delta_{\tilde{X_{i}}_{\bullet}|\tilde{Y_{1}}^{i}}(\phi_{i}^{*}D_{t})=\Delta_{\tilde{X_{i}}_{\bullet}|\tilde{Y_{1}}^{i}}(\tilde{f_{i}}^{*}({f_{i}}_{*}^{c}D_{t}))$ for each $t=t_{i} \in [\alpha_{i}, \beta_{i}]$. Then, for any $t \in [{\rm ord}_{\tilde{Y_{1}}}(\lVert \phi_{i}^{*}D \lVert), \mu]=[{\rm ord}_{\tilde{Y_{1}}}(\lVert \tilde{f_{i}}^{*}({f_{i}}^{c}_{*}D) \lVert), \mu]$, we obtain that 
\begin{align*}
{\Delta_{\tilde{X_{i}}_{\bullet}}(\phi_{i}^{*}D)}_{x_{1}=t}&=\Delta_{\tilde{X_{i}}_{\bullet}|\tilde{Y_{1}}^{i}}(\phi_{i}^{*}D-t \tilde{Y_{1}}^{i})=\Delta_{\tilde{X_{i}}_{\bullet}|\tilde{Y_{1}}^{i}}(\phi_{i}^{*}(D-t Y_{1})) \\
&=\Delta_{\tilde{X_{i}}_{\bullet}|\tilde{Y_{1}}^{i}}(\tilde{f_{i}}^{*}{f_{i}}^{c}_{*}(D-t Y_{1}))=\Delta_{\tilde{X_{i}}_{\bullet}|\tilde{Y_{1}}^{i}}(\tilde{f_{i}}^{*}{f_{i}}^{c}_{*}D-t \tilde{Y_{1}}^{i}) \\
&={\Delta_{\tilde{X_{i}}_{\bullet}}(\tilde{f_{i}}^{*}{f_{i}}^{c}_{*}D)}_{x_{1}=t}.
\end{align*}
Thus, $\Delta_{\tilde{X_{i}}_{\bullet}}(\phi_{i}^{*}D)=\Delta_{\tilde{X_{i}}_{\bullet}}(\tilde{f_{i}}^{*}{f_{i}}^{c}_{*}D)$.
\end{proof}

By Proposition \ref{prop:important lemma}, we can obtain the following corollary, which is a special case of \cite[Theorem C]{KL2}. 

\begin{cor}
Let $X$ be a Mori dream space of dimension $n$, $X_{\bullet}:X=Y_{0} \supset Y_{1} \supset \cdots \supset Y_{n}$ an admissible flag, and $D$ a big divisor on $X$ such that $D \in C_{i}$, where $C_{i}$ is a Mori chamber. Assume that $Y_{1} \cap {\rm Ud}(f_{i})=\emptyset$, where $f_{i}$ is a SQM of $X$ corresponding to $C_{i}$. Then, 
\begin{align*}
\Delta_{X_{\bullet}}(D)=\Delta_{X_{\bullet}}(P_{D})+a
\end{align*}
for some vector $a \in \mathbb{R}^{n}$, where $P_{D}$ is the positive part of the Zariski decomposition of $D$ in terms of MDS. 
\end{cor}

\begin{proof}
We use notations in Notation \ref{notation:Mori dream threefold} and Note \ref{note:set-up}. First, assume that $D$, $P_{D}$, and $N_{D}$ are all integral. By Proposition \ref{prop:important lemma}, $\Delta_{\tilde{X_{i}}_{\bullet}}(\phi_{i}^{*}D)=\Delta_{\tilde{X_{i}}_{\bullet}}(\tilde{f_{i}}^{*}{f_{i}}^{c}_{*}D)$. Let $a:=\nu_{{\tilde{X_{i}}}_{\bullet}}(\tilde{f_{i}}^{*}s_{N_{D}}) \in \mathbb{R}^{n}$ for the canonical section $s_{N_{D}} \in H^{0}(X_{i}, {f_{i}^{c}}_{*}N_{D})$. By using the same argument in the proof of Lemma \ref{lem:restriction}, we can obtain $\Delta_{\tilde{X_{i}}_{\bullet}}(\tilde{f_{i}}^{*}{f_{i}}^{c}_{*}D)=\Delta_{\tilde{X_{i}}_{\bullet}}(\tilde{f_{i}}^{*}{f_{i}}^{c}_{*}P_{D})+a$. Then, by Lemma \ref{lem:birational}, $\Delta_{X_{\bullet}}(D)=\Delta_{\tilde{X_{i}}_{\bullet}}(\phi_{i}^{*}D)=\Delta_{\tilde{X_{i}}_{\bullet}}(\tilde{f_{i}}^{*}{f_{i}}^{c}_{*}D)=\Delta_{\tilde{X_{i}}_{\bullet}}(\tilde{f_{i}}^{*}{f_{i}}^{c}_{*}P_{D})+a=\Delta_{\tilde{X_{i}}_{\bullet}}(\phi_{i}^{*}P_{D})+a=\Delta_{X_{\bullet}}(P_{D})+a$. \

In general, assume that $D$, $P_{D}$, and $N_{D}$ are $\mathbb{Q}$-divisors. Choose $m>>0$ such that $mD$, $mP_{D}$, and $mN_{D}$ are all integral. Then, by using the same argument as in the integral case to $mD$, $mP_{D}$, and $mN_{D}$, we obtain the desired result. 
\end{proof}

Now, we are ready to prove our main theorem. 

\begin{thm} \label{thm:slices}
Let $X$, $X_{\bullet}$, $D$, $I_{D}$, and $[\alpha_{i}, \beta_{i}]$ be as in Notation \ref{notation:Mori dream threefold} and Note \ref{note:set-up}. Then, for each $i=1, \dots, r$, there exist the linear function $l_{i}(t)=(l_{i}^{1}(t), \cdots, l_{i}^{n-1}(t))$ defined on each $[\alpha_{i}, \beta_{i}]$ such that 
\begin{align*}
\Delta_{X_{\bullet}}(D)_{x_{1}=t}=\Delta_{{Y_{1}}_{\bullet}}(P_{D_{t}}|_{Y_{1}})+l_{i}(t)
\end{align*} 
for all $t=t_{i} \in [\alpha_{i}, \beta_{i}]$. 
\end{thm}

\begin{proof}
Let $f_{i}$ be SQMs of $X$. Fix $t=t_{i} \in [\alpha_{i}, \beta_{i}]$. By Lemma \ref{lem:restriction} and Proposition \ref{prop:important lemma}, we obtain  
\begin{align*}
\Delta_{X_{\bullet}}(D)_{x_{1}=t}&=\Delta_{X_{\bullet}|Y_{1}}(D_{t})=\Delta_{\tilde{X_{i}}_{\bullet}|\tilde{Y_{1}}^{i}}(\phi_{i}^{*}D_{t})=\Delta_{\tilde{X_{i}}_{\bullet}|\tilde{Y_{1}}^{i}}(\tilde{f_{i}}^{*}({f_{i}}_{*}^{c}D_{t})) \\
&=\Delta_{\tilde{X_{i}}_{\bullet}|\tilde{Y_{1}}^{i}}(\tilde{f_{i}}^{*}({f_{i}}_{*}^{c}P_{D_{t}}))+l_{i}(t)=\Delta_{\tilde{Y_{1}}^{i}_{\bullet}}(\tilde{f_{i}}^{*}({f_{i}}_{*}^{c}P_{D_{t}})|_{\tilde{Y_{1}}^{i}})+l_{i}(t),
\end{align*}
where the second equality holds by Lemma \ref{lem:restricted pullback}, and the last equality comes from the continuity of the restricted Okounkov bodies and Proposition \ref{prop:restricted ample} since $\tilde{f_{i}}^{*}({f_{i}}_{*}^{c}P_{D_{t}})$ is nef. Now, consider the following commutative diagram. 
\begin{center}
\begin{displaymath}
\xymatrix{
\tilde{Y_{1}}^{i} \ar@{^{(}->}[r]^{i_{\tilde{Y_{1}}^{i}}} \ar@{->}[d]_{\phi_{i}|_{\tilde{Y_{1}}^{i}}} &
\tilde{X_{i}} \ar@{->}[d]^{\phi_{i}} \\
Y_{1} \ar@{^{(}->}[r]^{i_{Y_{1}}} & X }
\end{displaymath}
\end{center}
Since $Y_{1} \cap (\cup_{i \in I_{D}} {\rm Ud}(f_{i}))=\emptyset$, note that 
\begin{align*}
(\tilde{f_{i}}^{*}({f_{i}}_{*}^{c}P_{D_{t}}))|_{\tilde{Y_{1}}^{i}}&=i_{\tilde{Y_{1}}^{i}}^{*}(\tilde{f_{i}}^{*}({f_{i}}_{*}^{c}P_{D_{t}}))=i_{\tilde{Y_{1}}^{i}}^{*}(\phi_{i}^{*} f_{i}^{*}({f_{i}}_{*}^{c}P_{D_{t}})) \\
&=(\phi_{i}|_{\tilde{Y_{1}}^{i}})^{*}(f_{i}^{*}({f_{i}}_{*}^{c}P_{D_{t}}))|_{Y_{1}})=(\phi_{i}|_{\tilde{Y_{1}}^{i}})^{*}(P_{D_{t}}|_{Y_{1}}).
\end{align*}
Thus, $\Delta_{\tilde{Y_{1}}^{i}_{\bullet}}(\tilde{f_{i}}^{*}({f_{i}}_{*}^{c}P_{D_{t}})|_{\tilde{Y_{1}}^{i}})=\Delta_{{\tilde{Y_{1}}^{i}}_{\bullet}}((\phi_{i}|_{\tilde{Y_{1}}^{i}})^{*}(P_{D_{t}}|_{Y_{1}}))=\Delta_{{Y_{1}}_{\bullet}}(P_{D_{t}}|_{Y_{1}})$. Therefore, $\Delta_{X_{\bullet}}(D)_{x_{1}=t}=\Delta_{{Y_{1}}_{\bullet}}(P_{D_{t}}|_{Y_{1}})+l_{i}(t)$. 
\end{proof}

\begin{note} 
We give an expression of a Zariski decomposition of big divisors in terms of MDS (see the proof of \cite[Proposition 2.13]{O} for details). Let $g:X \rightarrow Y$ be contracting birational maps in Proposition \ref{prop:basic}-(2). Then, for any big divisor $D$ on $X$,  $P_{D}=g^{*}g^{c}_{*}D$ and $N_{D}=D-P_{D}$.
\end{note}

\begin{rmk} \label{rmk:extension} In this remark, we observe the meaning of Theorem \ref{thm:slices} and its extension to pseudo-effective divisors. 
\begin{enumerate} 
\item Theorem \ref{thm:slices} says that the problem on the descriptions of the Okounkov bodies of big divisors on $X$ associated to $X_{\bullet}$ is reduced to that on $Y_{1}$ associated to ${Y_{1}}_{\bullet}$.  
\item Theorem \ref{thm:slices} can be extended to the limiting Okounkov bodies of pseudo-effective divisors naturally. We leave the details to the reader. 
\item The first condition (${\rm Pic}(X)_{\mathbb{Q}}={\rm N^{1}}(X)_{\mathbb{Q}}$) in Definition \ref{defn:MDS} is not essential for Theorem \ref{thm:slices}. In fact, Theorem \ref{thm:slices} holds for any variety such that every big divisor has a decomposition in Proposition \ref{prop:Zariski decomposition}. 
\end{enumerate}
\end{rmk}

\begin{rmk} (Computations of Mori chamber decomposition)
For Theorem \ref{thm:slices}, we need to know what Mori chamber decomposition of a given $\overline{{\rm Eff}}(X)$ is. For computations of Mori chamber decomposition, see \cite[Chapter 3]{K}. 
\end{rmk}

As a byproduct of Theorem \ref{thm:slices}, we obtain the descriptions of the Okounkov bodies of big divisors on Mori dream threefolds. 

\begin{cor} \label{cor:Mori dream space}
Let $X$ be a Mori dream threefold, and $X_{\bullet}:X \supset S \supset C \supset \{p\}$ an admissible flag with $S \cap (\cup_{i \in I_{D}} {\rm Ud}(f_{i}))=\emptyset$, where $D$, $I_{D}$, and $[\alpha_{i}, \beta_{i}]$ are as in Notation \ref{notation:Mori dream threefold} and Note \ref{note:set-up}. Then, for each $i=1, \dots, r$, there exist the linear function $l_{i}(t)=(l_{i}^{1}(t), l_{i}^{2}(t))$ defined on each $[\alpha_{i}, \beta_{i}]$ such that 
\begin{align*}
\Delta_{X_{\bullet}}(D)=&\{(x_{1},x_{2},x_{3}) \in \mathbb{R}^{3}| \text{ }\text{\rm ord}_{S}(\lVert D \lVert) \le x_{1} \le \mu, \text{ for each } x_{1}=t_{i} \in [\alpha_{i}, \beta_{i}], \\
& l_{i}^{1}(t_{i}) \le x_{2} \le \mu_{t_{i}}+l_{i}^{1}(t_{i}), \text{ } \delta_{x_{2}}(t_{i}) \le x_{3} \le \delta_{x_{2}}(t_{i})+(P_{(P_{D_{t_{i}}}|_{S}-x_{2}C)}.C)\},
\end{align*}
where $\mu={\rm sup}\{t>0|\text{ }D-tS \in {\rm Big}(X)\}$, $\mu_{t}={\rm sup}\{\alpha>0|\text{ }P_{D_{t}}|_{S}-\alpha C \in {\rm Big}(S)\}$, $P_{D_{t_{i}}}|_{S}-x_{2}C=P_{(P_{D_{t_{i}}}|_{S}-x_{2}C)}+N_{(P_{D_{t_{i}}}|_{S}-x_{2}C)}$ is the Zariski decomposition in the usual sense, and $\delta_{x_{2}}(t_{i})={\rm ord}_{p}(N_{(P_{D_{t_{i}}}|_{S}-x_{2}C)}|_{C})+l_{i}^{2}(t_{i})$ is a linear function on each $[\alpha_{i}, \beta_{i}]$ for fixed $x_{2} \in [l_{i}^{1}(t_{i}), \mu_{t_{i}}+l_{i}^{1}(t_{i})]$.
\end{cor}

\begin{proof}
By Theorem \ref{thm:slices}, $\Delta_{X_{\bullet}}(D)_{x_{1}=t}=\Delta_{{S}_{\bullet}}(P_{D_{t}}|_{S})+l_{i}(t)$, where $D_{t}:=D-tS$. We are reduced to the surface case. By Example \ref{exmp:surface},  
\begin{align*}
\Delta_{S_{\bullet}}(P_{D_{t}}|_{S})=&\{(x_{2},x_{3}) \in \mathbb{R}^{2}|\text{ } 0 \le x_{2} \le \mu_{t}, \\
& {\rm ord}_{p}(N_{(P_{D_{t}}|_{S}-x_{2}C)}|_{C}) \le x_{3} \le {\rm ord}_{p}(N_{(P_{D_{t}}|_{S}-x_{2}C)}|_{C})+(P_{(P_{D_{t}}|_{S}-x_{2}C)}.C) \}, 
\end{align*}
where $\mu_{t}:={\rm sup}\{\alpha>0|\text{ }P_{D_{t}}|_{S}-\alpha C \in {\rm Big}(S)\}$ and $N_{(P_{D_{t}}|_{S}-x_{2}C)}$ is the negative part of $P_{D_{t}}|_{S}-x_{2}C$. Therefore, by summarizing them, we obtain the desired result. 
\end{proof}

\begin{caution} \label{caution:caution}
The description of $\Delta_{X_{\bullet}}(D)$ in Corollary \ref{cor:Mori dream space} may not holds without the assumption `$S \cap (\cup_{i \in I_{D}} {\rm Ud}(f_{i}))=\emptyset$'. For a counterexample, see Example \ref{exmp:blowing-up}.
\end{caution}

Corollary \ref{cor:Mori dream space} can be extended to the limiting Okounkov bodies of pseudo-effective divisors. Since the proof is similar to that of Corollary \ref{cor:Mori dream space}, we omit it here. For descriptions of the limiting Okounkov bodies of a surface, see \cite[Theorem 4.5]{CPW1}.

\begin{cor} \label{cor:pseudo-effective divisor}
Let $X$, $X_{\bullet}$, and $[\alpha_{i},\beta_{i}]$ be as in Corollary \ref{cor:Mori dream space}. Let $D$ be a psuedo-effective divisor, and $I_{D}:=\{i\text{ }|\text{ } D_{t} \in C_{i}\text{ for some $t \in [{\rm ord}_{S}(\lVert D \lVert), \mu]$}\}$, where $\mu={\rm sup}\{t>0|\text{ }D-tS \in {\rm Eff}(X)\}$. Then, for each $i=1, \dots, r$, there exist the linear function $l_{i}(t)=(l_{i}^{1}(t), l_{i}^{2}(t))$ defined on each $[\alpha_{i}, \beta_{i}]$ such that 
\begin{align*}
\Delta_{X_{\bullet}}^{\rm lim}(D)=&\{(x_{1},x_{2},x_{3}) \in \mathbb{R}^{3}| \text{ }\text{\rm ord}_{S}(\lVert D \lVert) \le x_{1} \le \mu, \text{ for each } x_{1}=t_{i} \in [\alpha_{i}, \beta_{i}], \\
& l_{i}^{1}(t_{i}) \le x_{2} \le \mu_{t_{i}}+l_{i}^{1}(t_{i}), \text{ } \delta_{x_{2}}(t_{i}) \le x_{3} \le \delta_{x_{2}}(t_{i})+(P_{(P_{D_{t_{i}}}|_{S}-x_{2}C)}.C)\},
\end{align*}
where $\mu_{t}={\rm sup}\{\alpha>0|\text{ }P_{D_{t}}|_{S}-\alpha C \in {\rm \overline{Eff}}(S)\}$, $P_{D_{t_{i}}}|_{S}-x_{2}C=P_{(P_{D_{t_{i}}}|_{S}-x_{2}C)}+N_{(P_{D_{t_{i}}}|_{S}-x_{2}C)}$ is the Zariski decomposition in the usual sense, and $\delta_{x_{2}}(t_{i})={\rm ord}_{p}(N_{(P_{D_{t_{i}}}|_{S}-x_{2}C)}|_{C})+l_{i}^{2}(t_{i})$ is a linear function on each $[\alpha_{i}, \beta_{i}]$ for fixed $x_{2} \in [l_{i}^{1}(t_{i}), \mu_{t_{i}}+l_{i}^{1}(t_{i})]$.
\end{cor}

\begin{rmk}
Note that \cite[Corollary 3.2]{KLM} describes the Okounkov bodies of big divisors on $X$, where $X$ is a smooth projective threefold such that $\overline{\rm Eff}(X)={\rm Nef}(X)$. In this case, it is clear that every effective divisor has a decomposition in Proposition \ref{prop:Zariski decomposition}, and that there are no undefined loci of SQMs of $X$ (so that we can choose any admissible flag). Therefore, by Remark \ref{rmk:extension}-(3), Corollary \ref{cor:Mori dream space} and \ref{cor:pseudo-effective divisor} can be seen as a generalization of \cite[Corollary 3.2]{KLM}. 
\end{rmk}

Now, we see some examples of Corollary \ref{cor:Mori dream space} and \ref{cor:pseudo-effective divisor}.

\begin{exmp} \label{exmp:blowing-up}
Let $X$ be the blowing-up of $\mathbb{P}^{3}$ at two points $p_{1}$ and $p_{2}$ with exceptional divisors $E_{1}$ and $E_{2}$. By \cite[Example 5.5]{AW}, $X$ is a Mori dream space. Let $\phi:X \rightarrow \mathbb{P}^{3}$ be the blowing-up map. Let $X_{\bullet}:X \supset S=E_{1} \cong \mathbb{P}^{2} \supset C \supset \{x\}$ be an admissible flag, where $C$ is any curve in $E_{1}$ and $x$ is any point in $C$. Let $d$ be the degree of $C$ and $L$ a line in $E_{1}$. Note that $C=dL$. For notational convenience, let $D_{t}:=D-tS$ for some pseudo-effective divisor $D$ with $t \in [\text{\rm ord}_{S}(\lVert D \lVert), \mu]$. \

Now, let $H$ be a hyperplane in $\mathbb{P}^{3}$ passing through $p_{1}$ and $p_{2}$. Furthermore, let $H_{1}:=\phi^{*}H-E_{1}$, $H_{2}:=\phi^{*}H-E_{2}$, and $H_{12}$ strict transform of plane passing through both $p_{1}$ and $p_{2}$. For descriptions of the (limiting) Okounkov bodies, we need to know Mori chamber decomposition of ${\rm Eff}(X)$. By \cite[Example 5.5]{AW}, $X$ has two Mori chambers, and denote nef parts of two Mori chambers by $N$ and $N'$, and each $N$ and $N'$ is generated by $\phi^{*}H$, $H_{1}$, $H_{2}$ and $H_{12}$, $H_{1}$, $H_{2}$, respectively. Moreover, let $\textcircled{\small 1}:=<\phi^{*}H, E_{1}, E_{2}>$, and $\textcircled{\small 2}:=<\phi^{*}H, H_{1}, E_{2}>$, where $<A,B,C>$ denotes the closed convex cone generated by $A$, $B$, and $C$ (FIGURE 1). First, consider a pseudo-effective divisor $D$ satisfying the condition on Corollary \ref{cor:pseudo-effective divisor} (or Corollary \ref{cor:Mori dream space}). 

\begin{figure}[h] 
\centering
\begin{tikzpicture}
\draw (0,0) node[anchor=north east]{$E_{2}$} -- (2,3.464101615) node[anchor=south]{$H_{12}$} -- (4,0) node[anchor=north west]{$E_{1}$} -- (0,0);
\draw (1,1.7320508075688) node[anchor=east]{$H_{1}$} -- (3,1.7320508075688) node[anchor=west]{$H_{2}$} -- (2,1.1547005383792515) node[anchor=north]{$\phi^{*}H$} --  (1,1.7320508075688);
\draw (2, 0.3) node{$\textcircled{\small 1}$};
\draw (1.2, 1.1) node{$\textcircled{\small 2}$};
\draw[loosely dashed] (2,1.1547005383792515) -- (4,0);
\draw[loosely dashed] (2,1.1547005383792515) -- (0,0);
\fill[gray] (1,1.7320508075688) -- (3,1.7320508075688) -- (2,1.1547005383792515) -- (1,1.7320508075688);
\fill[gray!60!white] (1,1.7320508075688) -- (2,3.464101615) -- (3,1.7320508075688) -- (1,1.7320508075688);
\end{tikzpicture}
\caption{Mori chamber decomposition of $\overline{{\rm Eff}}(X)$ (uppder side)}
\end{figure}
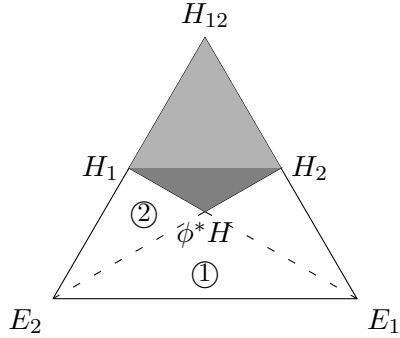

Let $D$ be a pseudo-effective divisor on $X$ such that $D \in \textcircled{\small 1}$. We may let $D=a\phi^{*}H+bE_{2}+cE_{1}$ with $a, b, c \ge 0$ with at least one nonzero $a, c$ (clearly, $\Delta_{X_{\bullet}}^{\rm lim}(bE_{2})=\{(0,0,0)\}$). Note that $c \le x_{1} \le a+c$. Also, for all $t \in [c, a+c]$, $D-tE_{1}=aH_{1}+(a+c-t)E_{1}+bE_{2}$, so $P_{D_{t}}=aH_{1}+(a+c-t)E_{1}$ and $N_{D_{t}}=bE_{2}$. So ${P_{D_{x_{1}}}}|_{E_{1}}=\frac{1}{d}(x_{1}-c)C$, $\mu_{x_{1}}=\frac{1}{d}(x_{1}-c)$ with $l_{i}(x_{1})=(0,0)$ for all $x_{1} \in [c, a+c]$. Therefore, we obtain
\begin{align*}
\Delta_{X_{\bullet}}^{\rm lim}(D)=&\{(x_{1},x_{2},x_{3}) \in \mathbb{R}^{3} | \text{ } c \le x_{1} \le a+c, \text{ } 0 \le x_{2} \le \frac{1}{d}(x_{1}-c), \\
& 0 \le x_{3} \le dx_{1}-d^{2}x_{2}-dc\}.
\end{align*}

Next, let $D$ be a pseudo-effective divisor on $X$ such that $D \in \textcircled{\small 2}$. We may let $D=aH_{1}+bE_{2}+c\phi^{*}H$ with $a, b, c \ge 0$ with at least one nonzero $a, c$. Note that $0 \le x_{1} \le c$. Also, for all $t \in [0,c]$, $D-tE_{1}=aH_{1}+bE_{2}+cH_{1}+(c-t)E_{1}=(a+c)H_{1}+(c-t)E_{1}+bE_{2}$. So $P_{D_{x_{1}}}=(a+c)H_{1}+(c-x_{1})E_{1}$, ${P_{D_{x_{1}}}}|_{E_{1}}=\frac{1}{d}(x_{1}+a)C$, $\mu_{x_{1}}=\frac{1}{d}(x_{1}+a)$, and $N_{D_{x_{1}}}=bE_{2}$: i.e., $l_{i}(x_{1})=(0,0)$ for all $x_{1} \in [0,c]$. Therefore, we obtain 
\begin{align*}
\Delta_{X_{\bullet}}^{\rm lim}(D)=&\{(x_{1},x_{2},x_{3}) \in \mathbb{R}^{3} | \text{ } 0 \le x_{1} \le c, \text{ } 0 \le x_{2} \le \frac{1}{d}(x_{1}+a), \\
& 0 \le x_{3} \le dx_{1}-d^{2}x_{2}+da\}.
\end{align*}

Now, let us see an example of Caution \ref{caution:caution}. Let $g$ be a non-trivial SQM of $X$. Since $E_{1} \cap {\rm Ud}(g) \neq \emptyset$ (\cite[Example 5.5]{AW}), any ample divisor does not satisfy the assumption ($S \cap (\cup_{i \in I_{D}} {\rm Ud}(f_{i}))=\emptyset$) in Corollary \ref{cor:Mori dream space}. Suppose that all ample divisors satisfy the description of Corollary \ref{cor:Mori dream space}. Then, for such $D$, we can easily obtain that 
\begin{align*}
\Delta_{X_{\bullet}}(D)=&\{(x_{1},x_{2},x_{3}) \in \mathbb{R}^{3} | \text{ } 0 \le x_{1} \le \mu, \text{ } 0 \le x_{2} \le \frac{1}{d}x_{1}+\frac{1}{d^{2}}(D.C), \\
& 0 \le x_{3} \le dx_{1}-d^{2}x_{2}+(D.C)\}.
\end{align*}
Fix an ample divisor $A$, and consider $H_{2}+\epsilon A$ with $\epsilon >0$. Clearly, $H_{2}+ \epsilon A$ is ample, so by definition, we obtain that 
\begin{align*}
\Delta_{X_{\bullet}}^{\rm lim}(H_{2}):=\lim_{\epsilon \rightarrow 0}\Delta_{X_{\bullet}}(H_{2}+\epsilon A)=&\{(x_{1},x_{2},x_{3}) \in \mathbb{R}^{3} | \text{ } 0 \le x_{1} \le 1, \text{ } 0 \le x_{2} \le \frac{1}{d}x_{1}, \\
& 0 \le x_{3} \le dx_{1}-d^{2}x_{2}\}.
\end{align*}
This is a contradiction since $\Delta_{X_{\bullet}}^{\rm lim}(H_{2})$ cannot be of full-dimensonal. 

\end{exmp}

\begin{exmp} \label{exmp:simple example}
Let $X$ be a $\mathbb{Q}$-factorial, normal hypersurface of any bidegree in $\mathbb{P}^{2} \times \mathbb{P}^{2}$. Then, by \cite[Theorem 1.1]{JCO}, $X$ is a Mori dream threefold with $\overline{{\rm Eff}}(X)=\overline{{\rm Mov}}(X)={\rm Nef}(X)$. More precisely, let $H_{i}=pr_{i}^{*}(\OO_{\mathbb{P}^{2}}(1))$ for $i=1, 2$, where $pr_{i}$ is the $i$-th projection. Then, 
\begin{align*}
\overline{{\rm Eff}}(X)=\overline{{\rm Mov}}(X)={\rm Nef}(X)=\mathbb{R}_{\ge 0}H_{1}+\mathbb{R}_{\ge 0}H_{2} \text{ {\rm (FIGURE 2)}}.
\end{align*} 
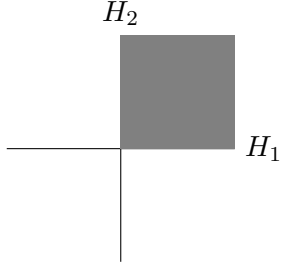
\begin{figure}[h] 
\centering
\begin{tikzpicture}
\draw (0,0) -- (0,1.5) node[anchor=south]{$H_{2}$};
\draw (0,0) -- (1.5,0) node[anchor=west]{$H_{1}$};
\draw (0,0) -- (0,-1.5);
\draw (0,0) -- (-1.5,0);
\fill[gray] (0,0) -- (0,1.5) -- (1.5,1.5) -- (1.5,0);
\end{tikzpicture}
\caption{$\overline{{\rm Eff}}(X)=\overline{{\rm Mov}}(X)={\rm Nef}(X)$ in ${\rm N^{1}}(X)_{\mathbb{R}}$}
\end{figure}
In this case, since small $\mathbb{Q}$-factorial modification is only the identity of $X$, there is no undefined locus of SQM on $X$: we can choose any admissible flag of $X$. Thus, we can easily obtain the limiting Okounkov bodies of any pseudo-effective divisors on $X$ for any admissible flags by using Corollary \ref{cor:pseudo-effective divisor} as in Example \ref{exmp:blowing-up}. We leave the details to the reader. 
\end{exmp}

\begin{exmp} \label{exmp:general hypersurface}
Let $X$ be a general hypersurface of bidegree $(d,e)$ in $\mathbb{P}^{1} \times \mathbb{P}^{3}$ with $d \le 3$. More precisely, let 
\begin{align*}
f=x_{0}^{d}f_{0}+x_{0}^{d-1}x_{1}f_{1}+\cdots +x_{1}^{d}f_{d}
\end{align*}
be the defining bihomogeneous polynomial of bidegree $(d,e)$ of $X$, where $x_{0}$, $x_{1}$ are coordinates on $\mathbb{P}^{1}$ and the $f_{i}$ are homogeneous forms of degree $e$ on $\mathbb{P}^{3}$. Note that $X$ is general in the sense that it is smooth and $(f_{0}= \cdots =f_{d}=0)$ is a smooth subvariety in $\mathbb{P}^{3}$ of codimension $d+1$ ($(f_{0}= \cdots =f_{d}=0)=\emptyset$ if $d=3$). As in Example \ref{exmp:simple example}, it is better to see \cite[Theorem 1.1]{JCO} for this example. Note that if $d$ satisfies $1 \le d \le 3$ or $e=1$, $X$ is a Mori dream threefold. Now, consider the following cases: \

(Case 1 : $d=3$ or $e=1$) In this case, $\overline{{\rm Eff}}(X)=\overline{{\rm Mov}}(X)={\rm Nef}(X)$. So there is only one SQM of $X$, which is the identity: we can take any pseudo-effective divisor and any admissible flag for Corollary \ref{cor:pseudo-effective divisor}. For descriptions of the limiting Okounkov bodies, as in Example \ref{exmp:simple example}, we leave the details to the reader. \

(Case 2 : $d=1$ with $e \ge 2$) In this case, by \cite[Theorem 1.1]{JCO},  
\begin{align*}
&\overline{{\rm Eff}}(X)=\mathbb{R}_{\ge 0}H_{1}+\mathbb{R}_{\ge 0}E \text{ {\rm (FIGURE 3)}}, \\
&\overline{\rm Mov}(X)={\rm Nef}(X)=\mathbb{R}_{\ge 0}H_{1}+\mathbb{R}_{\ge 0}H_{2},
\end{align*}
where $H_{i}$ is the pullback of hyperplanes by using $i$-th projection and $E=eH_{2}-H_{1}$. 
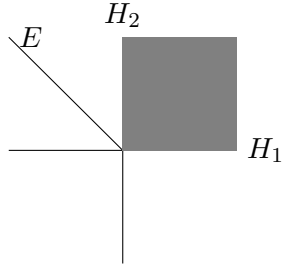
\begin{figure}[h] 
\centering
\begin{tikzpicture}
\draw (0,0) -- (0,1.5) node[anchor=south]{$H_{2}$};
\draw (0,0) -- (1.5,0) node[anchor=west]{$H_{1}$};
\draw (0,0) -- (0,-1.5);
\draw (0,0) -- (-1.5,0);
\draw (0,0) -- (-1.5,1.5) node[anchor=west]{$E$};
\fill[gray] (0,0) -- (0,1.5) -- (1.5,1.5) -- (1.5,0);
\end{tikzpicture}
\caption{$\overline{{\rm Eff}}(X)$ in ${\rm N^{1}}(X)_{\mathbb{R}}$}
\end{figure}
More precisely, if we let $f=x_{0}f_{0}+x_{1}f_{1}$ to be the defining bihomogeneous polynomial of $X$ in $\mathbb{P}^{1} \times \mathbb{P}^{3}$, then $X$ can be viewed as the blowing-up of $\mathbb{P}^{3}$ along a curve $C_{0}':=(f_{0}=f_{1}=0)$, where the blowing-up map $\pi$ is the second projection of $\mathbb{P}^{1} \times \mathbb{P}^{3}$ and $E$ is the exceptional divisor of this blowing-up. Also, since $\overline{{\rm Mov}}(X)={\rm Nef}(X)$, we can take any pseudo-effective divisors and any admissible flags of $X$.  Thus, as in Example \ref{exmp:blowing-up}, we can describe the limiting Okounkov bodies of any pseudo-effective divisors on $X$. Let us see a specific example. Let $f=x_{0}f_{0}+x_{1}f_{1}$ be the defining equation of $X$. Also, for a ruled surface $\pi:E \rightarrow C_{0}'$ with a section $C_{0}$, we assume that $({C_{0}}^{2}) \ge 0$. Then, by the proof of Lemma \ref{lem:ruled surface}, $\overline{{\rm Eff}}(E)={\rm Nef(E)}$. Fix an admissible flag $X_{\bullet}:X \supset E \supset C \supset \{p\}$, where $C$ is any curve of bidegree $(s_{1}, s_{2})$ in $E$ (i.e., $C=s_{1} C_{0}+s_{2} F$ with $(C_{0}^{2}) \ge 0$, $(F^{2})=0$ and $(C_{0}.F)=1$) and $p$ is any point in $C$. Let $\textcircled{\small 1} =<E,H_{2}>$, $\textcircled{\small 2}=<H_{1},H_{2}>$. First, let $D \in \textcircled{\small 1}$ be a pseudo-effective divisor on $X$. We may let $D=aE+bH_{2}$ with $a, b \ge 0$ with at least one nonzero $a$, $b$. It is clear that ${\rm ord}_{E}(\lVert D \lVert)=a$ and $\mu=a+\frac{b}{e}$. Let $A_{x_{1}}:=((x_{1}-a)H_{1}+(ae+b-x_{1}e)H_{2})|_{E}$ (we can describe $A_{x_{1}}=\alpha C_{0}+\beta F$ by using $(C_{0}^{2})$, $(H_{i}.C_{0})$ and $(H_{i}.F)$). Note that $\mu_{x_{1}}={\rm sup}\{\alpha>0 |\text{ } A_{x_{1}}-\alpha C \in {\rm Big}(E)\}$. Also, $(P_{(P_{D_{x_{1}}}|_{E}-x_{2}C)}.C)=(A_{x_{1}}.C)-x_{2}(s_{1}^{2}(C_{0}^{2})+2s_{1}s_{2})$. Therefore, we obtain
\begin{align*}
\Delta_{X_{\bullet}}^{\rm lim}(aE+bH_{2})=&\{(x_{1},x_{2},x_{3}) \in \mathbb{R}^{3}| \text{ } a \le x_{1} \le a+\frac{b}{e}, \text{ } 0 \le x_{2} \le \mu_{x_{1}}, \\
& 0 \le x_{3} \le (A_{x_{1}}.C)-x_{2}(s_{1}^{2}(C_{0}^{2})+2s_{1}s_{2})\}.
\end{align*}

Next, let $D \in \textcircled{\small 2}$ be a pseudo-effective divisor on $X$. We may let $D=aH_{1}+bH_{2}$ with $a, b \ge 0$ with at least one nonzero $a$, $b$. In this case, since $D$ is movable, ${\rm ord}_{E}(\lVert D \lVert)=0$. Also, since $D-x_{1}E=(x_{1}+a)H_{1}+(b-x_{1}e)H_{2}$, $\mu=\frac{b}{e}$. Let $A_{x_{1}}:=((x_{1}+a)H_{1}+(b-x_{1}e)H_{2})|_{E}$. As above, note that $\mu_{x_{1}}={\rm sup}\{\alpha>0|\text{ } A_{x_{1}}-\alpha C \in {\rm Big}(E)\}$. Also, $(P_{(P_{D_{x_{1}}}|_{E}-x_{2}C)}.C)=(A_{x_{1}}.C)-x_{2}(s_{1}^{2}(C_{0}^{2})+2s_{1}s_{2})$. Thus, we obtain
\begin{align*}
\Delta_{X_{\bullet}}^{\rm lim}(aH_{1}+bH_{2})=&\{(x_{1},x_{2},x_{3}) \in \mathbb{R}^{3}| \text{ } 0 \le x_{1} \le \frac{b}{e}, \text{ } 0 \le x_{2} \le \mu_{x_{1}}, \\
& 0 \le x_{3} \le (A_{x_{1}}.C)-x_{2}(s_{1}^{2}(C_{0}^{2})+2s_{1}s_{2})\}.
\end{align*}

(Case 3 : $d=2$ with $e \ge 2$) Finally, in this case, 
\begin{align*}
&\overline{{\rm Eff}}(X)=\overline{{\rm Mov}}(X)=\mathbb{R}_{\ge 0}H_{1}+\mathbb{R}_{\ge 0}(eH_{2}-H_{1}) \text{ {\rm (FIGURE 4)}}, \\
&{\rm Nef}(X)=\mathbb{R}_{\ge 0}H_{1}+\mathbb{R}_{\ge 0}H_{2},
\end{align*}
where $H_{1}$, $H_{2}$ are as above. 
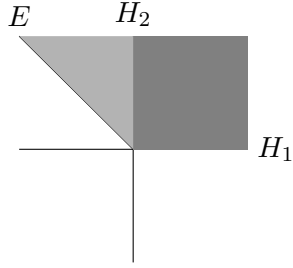
\begin{figure}[h] 
\centering
\begin{tikzpicture}
\draw (0,0) -- (0,1.5) node[anchor=south]{$H_{2}$};
\draw (0,0) -- (1.5,0) node[anchor=west]{$H_{1}$};
\draw (0,0) -- (0,-1.5);
\draw (0,0) -- (-1.5,0);
\draw (0,0) -- (-1.5,1.5) node[anchor=south]{$E$};
\fill[gray] (0,0) -- (0,1.5) -- (1.5,1.5) -- (1.5,0);
\fill[gray!60!white] (0,0) -- (0,1.5) -- (-1.5,1.5) -- (0,0);
\end{tikzpicture}
\caption{$\overline{{\rm Eff}}(X)$ in ${\rm N^{1}}(X)_{\mathbb{R}}$}
\end{figure}

Moreover, $X$ has two SQMs; one is the identity, and the other one with its undefined locus is disucssed in the proof of \cite[Theorem 1.1]{JCO}. More precisely, let $f=x_{0}^{d}f_{0}+x_{0}^{d-1}x_{1}f_{1}+\dots +x_{1}^{d}f_{d}$ be the defining bihomogeneous polynomial of $X$ on $\mathbb{P}^{1} \times \mathbb{P}^{3}$. Then, the undefined locus of such SQM is $Z(f_{0}, \dots, f_{d})$. Thus, for any admissible flag $X_{\bullet}:X \supset S \supset C \supset \{p\}$ with $S \cap Z(f_{0}, \dots, f_{d})=\emptyset$, we can describe the limiting Okounkov bodies of any pseudo-effective divisors on $X$ by Corollary \ref{cor:pseudo-effective divisor}. We leave the details to the reader. 
\end{exmp}

\end{section}

\begin{section} {Application}

In this section, we use Corollary \ref{cor:pseudo-effective divisor} to obtain conditions of rational polyhedrality of the limiting Okounkov bodies of pseudo-effective divisors on Mori dream threefolds. Before going on, we see some basic facts about Zariski chambers (see \cite{BKS} for details).

\begin{defn} \label{defn:Zariski chamber}
Let $S$ be a smooth projective surface. Let $P$ be a nef and big divisor on $S$. Define 
\begin{align*}
\Sigma_{P}=\{D \in {\rm Big}(S) \text{ } | \text{ } {\rm Neg}(D)={\rm Null}(P)\},
\end{align*}
where ${\rm Neg}(D)=\{C \text{ }| \text{ $C$ is an irreducible component of $N_{D}$}\}$, where $N_{D}$ is the negative part of $D$, and ${\rm Null}(P)$ is the set of irreducible curves on $S$ such that $(P.C)=0$. We call $\Sigma_{P}$ the Zariski chambers of $P$. 
\end{defn}

\begin{rmk} 
Note that $\Sigma_{P}$ is a convex cone that is neither open nor closed in general. Moreover, for nef and big divisors $P$ and $Q$, $\Sigma_{P} \bigcap \Sigma_{Q}=\emptyset$ if and only if $\Sigma_{P} \neq \Sigma_{Q}$. 
\end{rmk}

The importance of Zariski chambers come from the followinig Proposition \ref{prop:Zariski chamber} which is the main theorem in \cite{BKS}. 

\begin{prop} \label{prop:Zariski chamber}
Let $S$ be a smooth projective surface. Then, Zariski chambers give a locally finite decomposition of ${\rm Big}(S)$ such that the following holds: 
\begin{enumerate}[(1)]
\item The support of the negative part of the divisors in each chamber is constant. 
\item On each of the chambers, the volume function is given by a single quadratic polynomial. 
\item In the interior of each of the chambers, the stable base loci are constant. 
\end{enumerate}
\end{prop}

\begin{rmk} \label{rmk:linear}
On each Zariski chamber, the map taking a divisor to its negative part is linear. 
\end{rmk}

Now, by using Proposition \ref{prop:Zariski chamber} and Corollary \ref{cor:pseudo-effective divisor}, we obtain the following corollary.

\begin{cor} \label{cor:polyhedral}
Notation is as in Corollary \ref{cor:pseudo-effective divisor}. Let $D$ be any pseudo-effective divisor. Suppose the following conditions hold:  
\begin{enumerate}[(1)]
\item ${\rm Big}(S)$ has finitely many Zariski chambers. 
\item $\mu_{t}:={\rm sup}\{\alpha>0|\text{ }P_{D_{t}}|_{S}-\alpha C \in {\rm \overline{Eff}}(S)\}$ is piecewise linear with finitely many pieces.
\end{enumerate}
Then, $\Delta_{X_{\bullet}}^{\rm lim}(D)$ is rational polyhedral. \

In particular, if $S$ is of Picard number $1$, then $\Delta_{X_{\bullet}}^{\rm lim}(D)$ is rational polyhedral. 
\end{cor}

\begin{proof}
Since there are finitely many Mori chambers, we need to show the rational polyhedrality of $\Delta_{X_{\bullet}}^{\rm lim}(D)$ for each of them. Fix $i=1,\dots,r$. First, suppose that (1) and (2) hold. By Remark \ref{rmk:extension}-(2), $\Delta_{X_{\bullet}}^{\rm lim}(D)_{x_{1}=t}=\Delta_{S_{\bullet}}^{\rm lim}(P_{D_{t}}|_{S})+l_{i}(t)$ for each $t=t_{i} \in [\alpha_{i}, \beta_{i}]$. Since ${\rm Big}(S)$ has finitely many Zariski chambers, by Remark \ref{rmk:linear} and Example \ref{exmp:surface}, each $\Delta_{X_{\bullet}}^{\rm lim}(A)_{x_{1}=t_{i}}$ is rational polyhedral. However, since each $\mu_{t}$ is piecewise linear with finite pieces and $l_{i}(t)$ is linear, we are done. \

Now, let $S$ be of Picard number $1$. Then the condition (1) is clear. Also, since $P_{D_{t}}$ is linear on each Mori chamber and $S$ is of Picard number $1$, condition (2) holds. Therefore, $\Delta_{X_{\bullet}}^{\rm lim}(D)$ is rational polyhedral.
\end{proof}

\begin{exmp} \label{exmp:blowing-up rational}
Again, let $X$ be two points blowing-up of $\mathbb{P}^{3}$. Let $X_{\bullet}:X \supset S=E_{1} \cong \mathbb{P}^{2} \supset C \supset \{x\}$ be an admissible flag such that $C$ is a curve on $S$ and $x$ is a point in $C$. Then, by Corollary \ref{cor:polyhedral}, since $S$ is of Picard number $1$, $\Delta_{X_{\bullet}}^{\rm lim}(D)$ is rational polyhedral for any pseudo-effective divisor $D \in \textcircled{\small 1}$ or $\textcircled{\small 2}$ on $X$, where $\textcircled{\small 1}$ and $\textcircled{\small 2}$ are as in Example \ref{exmp:blowing-up}. 
\end{exmp}

\end{section}

\bibliographystyle{abbrv}

\begin{thebibliography}{00}

\bibitem{AW} K. Altmann, J. Wi\' sniewski, Polyhedral divisors of Cox rings, Michigan Math. J. 60 (2011), no. 2, 463–-480. 

\bibitem{BKS} T. Bauer, A. K\" uronya, and T. Szemberg, Zariski chambers, volumes, and stable base loci, J. Reine Angew. Math. 576 (2004), 209–-233. 

\bibitem{CHPW} S.R. Choi, Y. Hyun, J. Park, and J. Won, Asymptotic base loci via Okounkov bodies, preprint available at http://arxiv.org/pdf/1507.00817

\bibitem{CPW1} S.R. Choi, J. Park, and J. Won, Okounkov bodies associated to pseudoeffective divisors, preprint available at http://arxiv.org/pdf/1508.03922v3.pdf

\bibitem{HK} Y. Hu, and S. Keel, Mori dream spaces and GIT, Michigan Math. J. 48 (2000), 331–-348. 

\bibitem{K} S. Keicher, Algorithms for Mori dream spaces, PhD Thesis, Universit¨at T¨ubingen, 2014,
http://nbn-resolving.de/urn:nbn:de:bsz:21-dspace-540614.

\bibitem{KLM} A. K\" uronya, V. Lozovanu, and C. Maclean, Convex bodies appearing as Okounkov bodies of divisors, Adv. Math. 229 (2012), no. 5, 2622–-2639. 

\bibitem{KL} A. K\" uronya and V. Lozovanu, Infinitesimal Newton-Okounkov bodies and jet separation, preprint available at http://arxiv.org/pdf/1507.04339v1.pdf

\bibitem{KL2} A. K\" uronya and V. Lozovanu, Positivity of line bundles and Newton-Okounkov bodies, preprint available at http://arxiv.org/pdf/1506.06525v1.pdf

\bibitem{RL} R. Lazarsfeld, {\em Positivity in Algebraic Geometry I, II}, Ergebnisse der Mathematik und ihrer Grenzgebiete. 3. Folge. A Series of Modern Surveys in Mathematics [Results in Mathematics and Related Areas. 3rd Series. A Series of Modern Surveys in Mathematics], 48. Springer-Verlag, Berlin, 2004. 
         
\bibitem{RM} R. Lazarsfeld and M. Musta\c t\v a, Convex bodies associated to linear series, Ann. Sci. \'Ec. Norm. Sup\'er. (4) 42 (2009), no. 5, 783–-835. 

\bibitem{O} S. Okawa, On images of Mori dream spaces, to appear in Math. Ann. 

\bibitem{O1} A. Okounkov, Brunn-Minkowski inequality for multiplicities. 
Invent. Math. 125 (1996), no. 3, 405–-411.

\bibitem{O2} A. Okounkov, Why would multiplicities be log-concave? in The orbit method in geometry and physics, Progr. Math., 213, Birkh\"auser Boston, Boston, MA, 2003.  

\bibitem{JCO} J.C. Ottem, Birational geometry of hypersurfaces in products of projective spaces. Math. Z. 280 (2015), no. 1-2, 135–-148. 
14E05 

\end{thebibliography}

\end{document}